\documentclass{amsart}
\usepackage{graphicx}
\usepackage[utf8]{inputenc} 
\usepackage[T1]{fontenc}    
\usepackage{hyperref}       
\usepackage{url}            
\usepackage{booktabs}       
\usepackage{amsfonts}       
\usepackage{nicefrac}       
\usepackage{microtype}      
\usepackage{xcolor}         
\usepackage{amsmath}
\usepackage{amssymb}
\usepackage[foot]{amsaddr}
\usepackage{mathtools}
\usepackage{amsthm}
\usepackage{multirow}
\usepackage{soul}
\usepackage[capitalize,noabbrev]{cleveref}
\usepackage{natbib}
\usepackage{dsfont}
\usepackage{float}
\usepackage{combelow}
\usepackage{caption}
\usepackage{subcaption}
\usepackage{stfloats}
\usepackage{listings}
\usepackage{comment}
\usepackage{wrapfig}
\usepackage[usenames,svgnames,dvipsnames]{xcolor}
\usepackage{bm}
\usepackage{bbm}
\usepackage{epsfig}
\usepackage{subcaption}
\usepackage{epstopdf}
\usepackage{latexsym}
\usepackage{enumerate}
\usepackage{tabularx}
\usepackage[left=1in, right=1in]{geometry}
\newtheorem{thm}{Theorem}[section]

\newtheorem{rem}[thm]{Remark}
\newtheorem{asm}[thm]{Assumption}
\numberwithin{equation}{section}
\usepackage[capitalize,noabbrev]{cleveref}
\usepackage{todonotes}

\newcommand*\xbar[1]{%
  \hbox{%
    \vbox{%
      \hrule height 0.5pt 
      \kern0.5ex
      \hbox{%
        \kern-0.1em
        \ensuremath{#1}%
        \kern-0.1em
      }%
    }%
  }%
}

 \title{Group Evolving Dynamics in Biased Condition: Modeling and Analysis}
\author{Samit~Ghosh}
\email[SG]{s2ghosh@odu.edu}

\address[SG]{Department of Mathematics and Statistics, Old Dominion University, Norfolk, VA 23529, USA}
\date{\today}

\begin{document}
\begin{abstract}
We propose a dynamical model for group formation and switching behavior in systems where each group competes for members through attraction functions that are inversely proportional to their current sizes. This attraction is modulated by group-specific bias terms, which can reflect social, economic, or reputational advantages. New entrants choose groups probabilistically based on these weighted attraction scores. We derive the conditions under which the system converges to a stationary equilibrium, where group proportions remain stable over time. The model exhibits rich nonlinear behavior, especially under varying bias strengths and inverse preference intensities. We analyze equilibrium conditions both theoretically and via simulations.\\[0.05in]
\end{abstract}
\maketitle
\section{\centering Introduction}\label{sec:Intro}
Understanding how individuals affiliate with groups is central to modeling a wide range of complex social, economic, and biological systems. From the formation of political coalitions and religious sects to the growth of online communities and consumer market segmentation, group dynamics shape collective behavior in profound and often unpredictable ways. These dynamics are rarely static; individuals frequently switch affiliations in response to factors such as group size, social influence, peer behavior, or perceived utility. Capturing such switching behavior is essential for explaining key social phenomena such as the persistence of diversity, the dominance of certain ideologies or brands, or the sudden collapse of seemingly stable configurations.

In many real-world settings, individuals may actively prefer smaller or less saturated groups in order to avoid overcrowding, competition, or dilution of identity. This behavior has been observed in organizational settings, niche consumer markets, and ecological systems, where the attractiveness of a group is often inversely related to its size. Such observations motivate models that incorporate \textit{inverse proportionality}, in which the likelihood of an individual joining a group decreases with that group's size—thereby promoting underrepresented or emerging groups. At the same time, individuals may exhibit \textit{group-specific biases} due to ideological, cultural, or historical affinities, further complicating the dynamics of affiliation and identity.

Our model seeks to formalize this interplay through a probabilistic group selection mechanism that blends mutual attraction (based on group proportions) with a tunable \textit{size-bias parameter}. This framework allows us to study not only equilibrium outcomes and long-run convergence behavior, but also the rich \textit{transient dynamics} that emerge when individuals repeatedly enter, leave, and reshape the system. Such a model has broad implications for understanding polarization, segregation, viral adoption, echo chambers, and coordination challenges in dynamic multi-agent environments.

The modeling of group formation and social dynamics has long attracted attention in mathematical sociology, statistical physics, and economics. Schelling's seminal work on spatial segregation \citep{schelling1971} illustrated how even mild individual preferences can lead to stark macro‑level segregation, laying the groundwork for many later models of emergent social structure. DeGroot's consensus model \citep{degroot1974} offered a formal approach to how beliefs evolve over time through repeated local averaging, establishing the foundation for a wide class of models in \textit{opinion dynamics}. To address the limitations of linear models like DeGroot's, Hegselmann and Krause introduced a \textit{bounded confidence} framework \citep{hegselmann2002}, where agents only interact with others whose opinions lie within a certain range, capturing how fragmentation and polarization can arise.

From a network science perspective, the emergence of \textit{scale-free networks} as described by Barabási and Albert \citep{barabasi1999} highlights how preferential attachment fosters hub formation and rich-get-richer dynamics, motivating the use of size-biasing in models of group affiliation. From a strategic and evolutionary standpoint, Taylor and Jonker’s replicator dynamics \citep{taylor1978} model how strategy distributions evolve based on relative payoff feedback, endowing our conceptualization of group "fitness" with adaptive underpinnings. Similarly, the concept of \textit{informational cascades} introduced by Bikhchandani et al. \citep{bikhchandani1992} explains how popular trends or ideologies emerge through sequential decision-making, even with incomplete information.

Recent work has further advanced these themes. Zabarianska and Proskurnikov (2024) extended opinion dynamics models to consider generalized, time-varying conformity weights, showing how group pressure robustly drives consensus under certain conditions \citep{zabarianska2024}. Sampson et al. (2024) introduced a hypergraph-based framework where both individual and small group (“polyadic”) opinions interact, leading to oscillatory and excitable dynamics akin to social fads \citep{sampson2024}. In computational modeling, Pasimeni et al. (2025) used agent-based simulations in a 2D opinion–space to explore clustering via attraction–repulsion, offering fresh insight into socio-spatial group formation \citep{pasimeni2025}. Kinetic, data-driven approaches (Albi et al. 2025) proposed models that couple opinion evolution with diffusion and control mechanisms to maintain diversity in social media-like systems \citep{albi2025}. A recent holistic review by Peralta et al. (2023) integrates multiple decision-making stages—from belief formation to diffusion to updating—providing a comprehensive framework for modeling group decisions in networks \citep{peralta2023}. Moreover, models incorporating dynamic interaction structures and radical agents have demonstrated the complex interplay between micro-level personal influence and macro-level polarization \citep{frontiers2025}.

In this work, we synthesize these diverse perspectives to propose a model that integrates mutual attraction, inverse group-size effects, and probabilistic decision-making. By combining elements from opinion dynamics, evolutionary game theory, network growth, and recent polyadic and kinetic formulations, our model seeks to characterize both equilibrium and transient behaviors in systems with dynamic group affiliations. We aim to provide theoretical insights into how social diversity is maintained, how dominant groups emerge, and how small perturbations can yield cascading shifts in group composition.

The proposed model of dynamic group formation builds upon several established streams of research, including reinforced stochastic processes, population dynamics, opinion formation, and probabilistic choice theory. We discuss these connections in the context of Equations~\eqref{eq:grp_prop}-~\eqref{eq:choosing_prob}.


Equation~\eqref{eq:grp_prop} defines the relative population proportion $\pi_k(t) = n_k(t)/N(t)$ of group $k$ at time $t$, analogous to formulations in evolutionary game theory and replicator dynamics \citep{sandholm2010population,taylor1978}. Similar normalization appears in generalized Pólya urn processes \citep{bandyopadhyay2018negative,kaur2019negatively, pemantle2007survey}, where proportions evolve through stochastic reinforcement updates. These frameworks provide the mathematical foundation for analyzing convergence and equilibrium stability of composition-based systems.


The mutual attraction matrix $M_{ij}(t)$ in Equation~\eqref{eq:Mutual_a_mat} introduces nonlinear pairwise coupling between groups. This mechanism blends cooperative and competitive interactions, reflecting both similarity and overlap among groups. Comparable constructs arise in sociophysics and opinion dynamics, notably in the bounded-confidence models of \citep{deffuant2000mixing} and \citep{hegselmann2002opinion}, where attraction strength depends on proximity of opinions. Normalization by $\max\{\pi_i,\pi_j\}$, however, distinguishes the present model by scaling interaction intensity relative to group dominance, a feature absent in classical formulations. The structure of $M_{ij}(t)$ also parallels affinity kernels used in social influence networks \citep{friedkin2011social} and modularity-based community models \citep{newman2006modularity}.


The cumulative attraction $\theta_k(t)$ in Equation~\eqref{eq:cum_attraction} aggregates the influence of all active groups on group $k$. This notion is conceptually related to potential fields in collective behavior models \citep{mas2013differentiation}, as well as to total influence scores in network-based formulations of social adaptation \citep{friedkin2011social}. The resulting attraction potential Equation ~\eqref{eq:Attraction_Potential}
introduces a tunable inverse-preference term controlled by $\beta$, which modulates the relative attractiveness of small versus large groups. This mechanism is reminiscent of negative reinforcement in generalized urn models \citep{kaur2019negatively} and of negative frequency-dependent selection in evolutionary theory \citep{christie2023nfds}, both of which promote diversity and equilibrium coexistence.


The choice probability in Equation~\eqref{eq:choosing_prob} follows a softmax form, where entrants select group $k$ proportionally to its attraction potential with additive stochastic perturbations $\epsilon_k \sim \mathcal{N}(\mu,\sigma^2)$. This structure is closely related to discrete-choice and logit models in econometrics \citep{mcfadden1974conditional}, and to quantal response equilibria in behavioral game theory \citep{mckelvey1995quantal}. Logit or softmax response rules are widely used in evolutionary population games to capture bounded rationality and exploration \citep{blume1993statistical, sandholm2010population}. The additive Gaussian noise ensures non-degenerate probabilities and continuous adaptation over time.


The power-law dependence $\pi_k^{-\beta}(t)$ links the model to anti-preferential attachment mechanisms studied in complex network growth, where new nodes favor connections to low-degree nodes rather than to hubs \citep{siew2020inverse}. For $\beta > 0$, this formulation produces negative feedback that stabilizes population proportions, analogous to negative frequency-dependent selection in ecology \citep{christie2023nfds}. Conversely, $\beta < 0$ induces positive reinforcement and potential monopolization dynamics, leading to rich nonlinear behaviors akin to bifurcations observed in reinforced stochastic processes \citep{bandyopadhyay2018negative}.


Overall, the present formulation integrates ideas from multiple domains: (i) population-based reinforcement and replicator dynamics, (ii) mutual-attraction coupling as in social and opinion models, (iii) inverse-preference mechanisms inspired by anti-preferential attachment, and (iv) probabilistic softmax choice inspired by discrete-choice theory. This synthesis allows the model to capture emergent equilibrium diversity and switching behavior within competitive group systems.

The paper is structured as follows: The \Cref{sec:Intro} outlines the problem motivation, reviews relevant literature, and states the main contributions. The \Cref{sec:Model} formally defines the proposed group dynamics model, detailing the underlying assumptions and mathematical framework. In \Cref{sec:stab}, we analyze the equilibrium properties and conditions for convergence of group proportions. The \Cref{sec:analy}  offers deeper theoretical insights into model behavior, including bifurcation phenomena and parameter sensitivity. The \Cref{sec:simu}  presents computational experiments that illustrate the dynamics under various scenarios and validate theoretical results. The \Cref{sec:summ} recaps the main findings and their implications. 
Finally, in \Cref{sec:lim} we claim our limitation and improvement of our model that is our next foucs.

\section{\centering Model Description}\label{sec:Model}

We consider a system of $K$ groups, where individuals arrive sequentially and probabilistically join one of the groups based on mutual attraction and a size-biasing mechanism.


Let $n_k(t)$ denote the number of individuals in group $k$ at time $t$, and let $N(t) = \sum\limits_{k=1}^K n_k(t)$ be the total number of individuals. The proportion of group $k$ is defined as:
\begin{equation}\label{eq:grp_prop}
    \pi_k(t) = \frac{n_k(t)}{N(t)}.
\end{equation}


Using \Cref{eq:grp_prop}, we define the mutual attraction (will be analyzed in \Cref{prop:concavity}) between group $i$ and group $j$ as:
\begin{equation}\label{eq:Mutual_a_mat}
    M_{ij}(t) = \frac{\pi_{i}^2(t) + \pi_{j}^2(t) - \pi_i(t)\pi_j(t)}{\max\{\pi_i(t), \pi_j(t)\}}.
\end{equation}

The denominator $\max\{\pi_i(t),\pi_j(t)\}$ equals zero only if 
$\pi_i(t) = \pi_j(t) = 0$, i.e., when both groups are empty. 
In this case, the numerator is also zero, leading to the indeterminate form $0/0$. Now consider

\medskip
\begin{equation*}
M_{ij}(t) \;=\;
\begin{cases}
\dfrac{\pi_{i}^2(t) + \pi_{j}^2(t)- \pi_i(t)\pi_j(t)}{\max\{\pi_i(t), \pi_j(t)\}}, 
& \text{if } \max\{\pi_i(t),\pi_j(t)\} > 0, \\[10pt]
0, & \text{if } \pi_i(t)=\pi_j(t)=0.
\end{cases}
\end{equation*}
Intuitively, two empty groups exert no mutual attraction. For epsilon smoothing of above formulation avoids case distinctions and is numerically stable.
\begin{equation*}
M_{ij}(t) \;=\;
\frac{\pi_{i}^2(t) + \pi_{j}^2(t) - \pi_i(t)\pi_j(t)}{\max(\pi_i(t), \pi_j(t)) + \varepsilon}, 
\qquad \varepsilon>0 \text{ small}.
\end{equation*}
$M_{ij}(t)$ only for the active set
$\{k : \pi_k(t) > 0\}$ and set $M_{ij}(t)=0$ whenever both groups are inactive. 
This is conceptually clean if the model allows birth processes for new groups. If one group is empty while the other is non-empty 
(e.g., $\pi_i(t)=0$, $\pi_j(t)>0$), then
\begin{equation*}
M_{ij}(t) = \pi_j(t).
\end{equation*}
Thus, an empty group still ``feels'' the presence of the other group through its cumulative attraction $\theta_k(t)$ ( \Cref{eq:cum_attraction}), and due to $\varepsilon>0$ in 
the choice model, it may still start attracting entrants. This formulation captures the synergy between larger groups while penalizing overlaps. Now, the \Cref{eq:Mutual_a_mat} helps us to define cumulative attraction of group $k$ is:
\begin{equation}\label{eq:cum_attraction}
 \theta_k(t) = \sum_{j=1}^K M_{kj}(t).   
\end{equation}
Then, \Cref{eq:cum_attraction} is key to define the attraction potential of group $k$ as:
\begin{equation}\label{eq:Attraction_Potential}
   a_{ki}(t) = \theta_k(t) \cdot \pi_{k}^{-\beta}(t) 
\end{equation}
where $\beta \geq 0$ (can be less than $0$ for richer model See \Cref{sec:simu}) is a tunable parameter controlling the preference for smaller groups. Finally, garbing the idea from \Cref{eq:Attraction_Potential} and considering the softmax property that the probability of a new entrant chooses group $k$ is:
\begin{equation}\label{eq:choosing_prob}
    p_{ik}(t) = \frac{a_{ki}(t) + \epsilon_k}{\sum_{j=1}^K (a_{ij}(t) + \epsilon_j)}
\end{equation}
where \(\epsilon_k \sim \mathcal{N}(\mu, \sigma^2) \quad \text{truncated to } [0, \infty)\) ensures that all groups maintain a nonzero probability.

We define the mutual attraction between group \( i \) and group \( j \) at time \( t \) as \Cref{eq:Mutual_a_mat}. This formulation captures synergy between group proportions while penalizing excessive overlap. A mathematical analysis of the curvature of \( M(x, y) \), with \( x = \pi_i \) and 
\( y = \pi_j \), \Cref{prop:concavity} shows that its Hessian matrix is degenerate: its determinant is zero at every interior point. This degeneracy means that \( M(x, y) \) is convex but not strictly convex, with one direction of curvature and one “flat” direction. Geometrically, the surface of \( M(x, y) \) has a ridge-like structure \Cref{fig:Mu_Attra}. As a result, the dynamics of group growth are \emph{path-dependent} and sensitive to initial conditions: small perturbations in early group composition can shift the long-term evolution toward different equilibrium configurations.

\begin{wrapfigure}{r}{0.5\textwidth}
\centering
\includegraphics[width=0.5\textwidth]{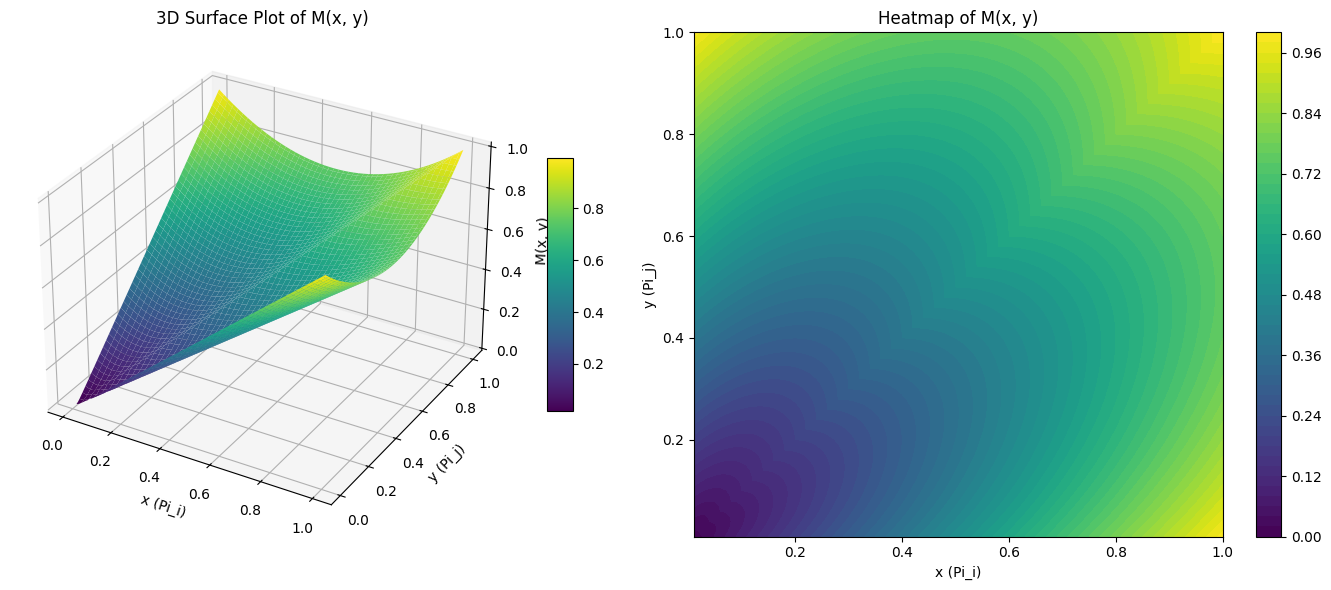}
\caption{Mutual Attraction for group formation}
\label{fig:Mu_Attra}
\end{wrapfigure}

Moreover, the ridge-like geometry of $M(x,y)$ (\Cref{fig:Mu_Attra}), which follows from the rank-one
degeneracy of its Hessian (\Cref{prop:concavity}), implies the presence of a neutral (flat) direction through every interior point of the state space. Rather than generating multiple metastable states, this structural feature produces neutral drift: curvature is positive in one direction but vanishes in another, so the dynamics do not pull trajectories toward a unique interior equilibrium. As a consequence, the system becomes path-dependent. Small differences in initial group proportions can direct the trajectory along different flat directions before boundary forces dominate, ultimately steering the system toward distinct boundary-dominated equilibrium configurations (see \Cref{sec:stab}).

While the degeneracy of the Hessian does not create multiple stable interior attractors, it does endow the system with flexibility in its transient behavior. Group proportions may remain diverse over intermediate time scales as trajectories drift along the flat direction of $M$. This neutral-drift mechanism allows the model to reproduce realistic patterns of temporarily stable coexistence and composition variability commonly observed in social, biological, and economic group-formation processes. However, in the long run, the weak convexity of $M$ ensures that the system does not settle into a single interior fixed point but instead evolves toward boundary equilibrium outcomes determined by initial conditions.

Let us denote $x = \pi_i(t)$ and $y = \pi_j(t)$, assuming $x \le y$. Then the equation \Cref{eq:Mutual_a_mat} reduces to
\[
M(x, y) = \frac{x^2 + y^2 - xy}{\max\{x,y\}}
         = \frac{x^2}{y} + y - x.
\]
The function $M(x,y)$ rewards larger group sizes through the quadratic terms $x^2 + y^2$, while the mixed term $-xy$ penalizes perfectly symmetric dominance. This structure implies that mutual attraction is highest when group sizes are uneven yet substantial, thereby encouraging heterogeneity rather than convergence to perfectly symmetric configurations.

The monotonicity properties of $M$ reflect this tension: the attraction increases with $x$ when $x < \frac{y}{2}$ and with $y$ when $x < y$, but these effects weaken as the two proportions approach one another. The second-order analysis (\Cref{prop:concavity}) shows that the Hessian of $M$ is positive semidefinite but degenerate, with rank one. Thus, $M$ is weakly convex with a flat direction through every interior point of the domain. Rather than producing saddle-like behavior, this geometry creates a neutral-gradient direction along which the dynamics exhibit path dependence. Small early variations in relative group sizes can steer the system along different flat trajectories before boundary effects dominate.

Consequently, the mutual attraction function $M$ encourages temporary diversity in group proportions and allows for slow drift along flat directions of the landscape, even though no interior equilibrium exists. Over longer horizons, the weak convexity of $M$ ensures that the dynamics ultimately move toward boundary-dominated equilibrium outcomes determined by initial conditions (see \Cref{sec:stab}).

\section{\centering Stability}\label{sec:stab}
To understand the evolution of group proportions over time, we investigate the stability and convergence properties of the underlying dynamics. The process describing how individuals join different groups based on mutual attraction and size bias evolves according to a stochastic update rule. Importantly, this process is Markovian \Cref{thm:Markov} in nature: the probability of a new entrant joining a group depends only on the current group sizes (or proportions), and not on the past history. This memoryless property makes the process analytically tractable and allows for the study of long-term behavior through standard tools from stochastic processes and dynamical systems.

\begin{thm}\label{thm:Markov}
The process $\{\pi(t)\}_{t \geq 0}$ is a Markov process.
\end{thm}
\begin{proof} 
A process $\{\pi_{t}\}$ is Markovian if:
$
\mathbb{P}(\pi_{t+1} | \pi_{t}, \pi_{t-1}, \dots, \pi_{0}) = \mathbb{P}(\pi_{t+1} | \pi_{t}).
$
Here, $\pi_{t} = \pi(t) \in \Delta^{K-1}$, the $(K-1)$-dimensional probability simplex. The update rule:
\[
\pi_k(t+1) = \pi_k(t) + \frac{p_{k}(\pi(t)) - \pi_k(t)}{N(t) + 1}
\]
where $p_{k}(\pi(t))$ is the the probability of choosing $k$ group at time $t$, depends only on $\pi(t)$.
Therefore,
\[
\mathbb{P}(\pi(t+1) | \pi(t), \pi(t-1), \dots) = \mathbb{P}(\pi(t+1) | \pi(t)),
\]
and hence $\{\pi(t)\}$ is a (deterministic) Markov process.
\end{proof}

\begin{wrapfigure}{r}{0.5\textwidth}
\centering
\includegraphics[width=0.5\textwidth]{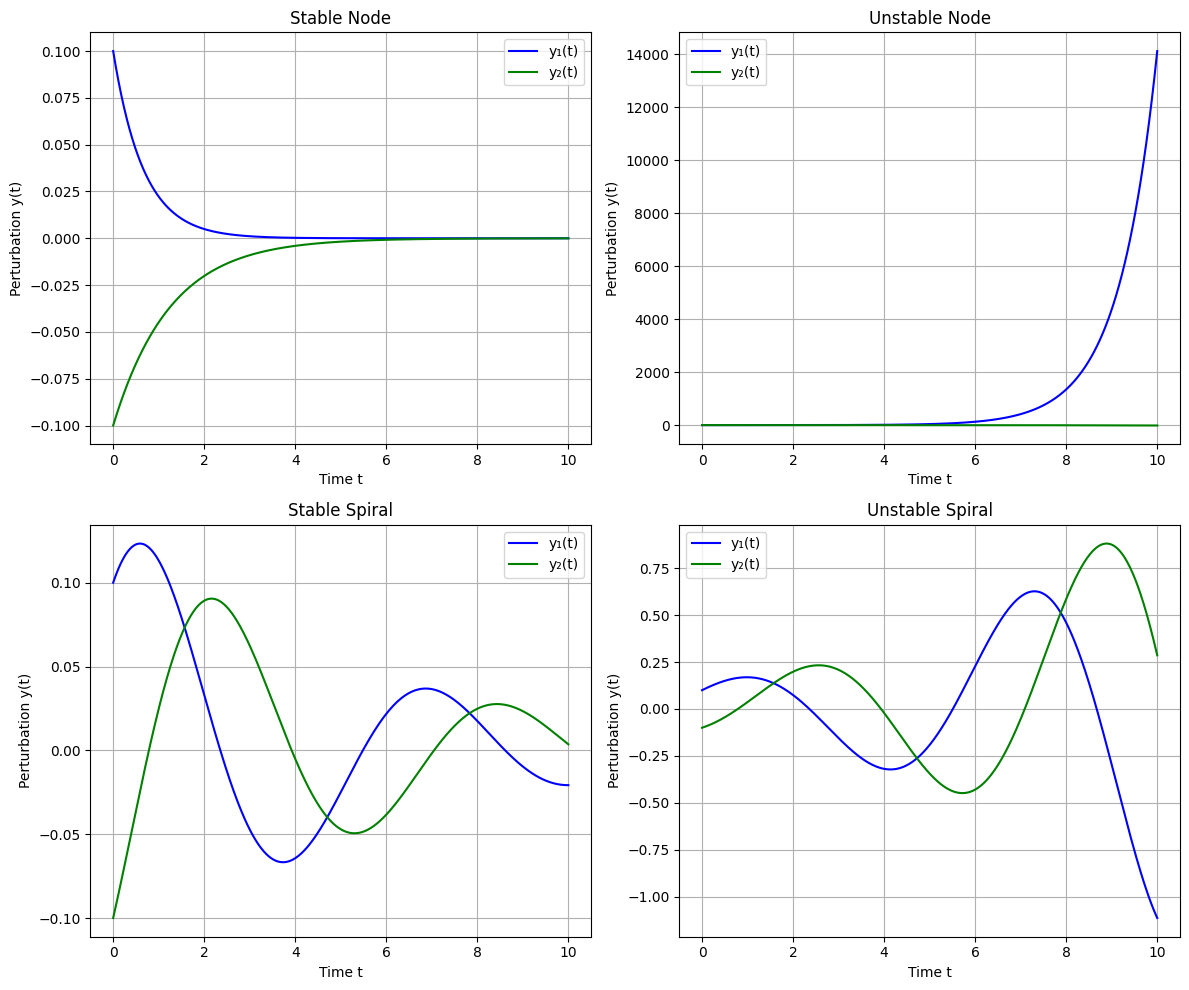}
\caption{Analytical solution around fixed point.}
\label{fig:A_sol}
\end{wrapfigure}

Graphically, the behavior described in \Cref{thm:analysis} of the system depends on the spectrum of $(J - I)$ (see \Cref{eq:analytic_Sol}): Like for stable Node: If all eigenvalues satisfy $\mathrm{Re}(\lambda_i) < 0$, then all perturbations decay exponentially and $\pi$ converges smoothly to $\pi^*$. The group proportions stabilize over time. For unstable Node: If any eigenvalue has $\mathrm{Re}(\lambda_i) > 0$, perturbations grow and the system diverges from equilibrium. Group proportions can shift dramatically, leading to instability. Similarly for stable Spiral: If some eigenvalues are complex with negative real parts, the system exhibits damped oscillations. Group proportions oscillate around equilibrium but eventually settle. Finally for unstable Spiral: If any complex eigenvalue has a positive real part, the system undergoes growing oscillations, spiraling away from equilibrium. These different regimes are illustrated in Figure~\ref{fig:A_sol}, which show representative trajectories of group proportions $\pi_k(t)$ under various spectral conditions.

Under symmetric bias \Cref{thm:equilibrum}—when all groups have the same attraction strength beyond size effects—the system tends toward a uniform distribution over groups. That is, as time progresses and more individuals join, the proportion of members in each group stabilizes at \( \frac{1}{K} \), where \( K \) is the total number of groups. This outcome aligns with intuition: if all groups are equally attractive and the only differentiating factor is size, then a balance emerges through random fluctuations averaging out over time.


When the bias is asymmetric \Cref{thm:equilibrum}—that is, when different groups are intrinsically more or less attractive due to differences in their bias parameters \( \epsilon_k \)—the equilibrium shifts accordingly. The system no longer converges to uniformity but instead to a biased equilibrium where the final group proportions reflect the relative strengths of these biases. In this case, larger \( \epsilon_k \) values correspond to higher limiting proportions \( \pi_k^* \), and the population distribution stabilizes in a way that reinforces the preferential nature of the dynamics.

In summary, the system exhibits stable and predictable behavior: it converges to a uniform distribution under symmetric bias, and to a deterministic, biased equilibrium when biases differ across groups. This analysis provides theoretical assurance that the modeled dynamics are well-behaved and meaningful in long-term scenarios, making them suitable for applications.


\begin{thm}\label{thm:equilibrum}
From  the described dynamics in \Cref{sec:Model}, Equation \Cref{eq:grp_prop} tends to balance out over time.
i.e,
\[
\Pr ( \lim_{t \to \infty} \pi_k(t) \overset{\text{a.s.}}{\to}  \frac{1}{K}) = 1, \quad \text{for all } k = 1, \dots, K
\]
That implies to the system has a fixed point for Equation \Cref{eq:choosing_prob}:
\begin{enumerate}
    \item With symmetric bias (\( \epsilon_k = \epsilon \)), the system converges to a uniform distribution: \( \pi_k^* = \frac{1}{K} \).
    \item With asymmetric bias, the system converges to a biased equilibrium where \( \pi_k^* \) reflects the relative magnitudes of \( \epsilon_k \).
    \item The dynamics are stable and globally convergent in both cases.
\end{enumerate}
\end{thm}


\begin{rem}\label{rm:global_stable}
From the symmetric dynamics: $\frac{d\pi_k}{dt} = p_{k}(\Pi) - \pi_k$. If $\pi_k > \frac{1}{K}$, then $\pi_k^{-\beta}$ decreases, so $p_{k}(\Pi) < \pi_k \Rightarrow \frac{d\pi_k}{dt} < 0$, which reduces $\pi_k$.
If $\pi_k < \frac{1}{K}$, then $\pi_k^{-\beta}$ increases, so $p_{k}(\Pi) > \pi_k \Rightarrow \frac{d\pi_k}{dt} > 0$, which increases $\pi_k$.
Therefore, the equilibrium is globally asymptotically stable.
\end{rem}
\section{\centering Analysis}\label{sec:analy}


Theorem~\ref{thm:analysis} formalizes the evolution of group proportions $\pi(t)$ as a 
\emph{Robbins--Monro stochastic approximation} \citep{Benaim1996,KushnerYin2003}. 
Each group’s proportion is updated toward the expected probability that a new entrant joins that group, 
with a step size $\gamma_t = 1/(N_k(t)+1)$ that decreases over time. Under the standard stochastic approximation assumptions—namely, 
\emph{step size conditions} (Assumption~4.2), \emph{boundedness of iterates} (Assumption~4.3), 
\emph{Lipschitz continuity of the drift function} (Assumption~4.4), and \emph{bounded martingale noise} (Assumption~4.5)—the sequence $\pi(t)$ converges \emph{almost surely} to a stable equilibrium $\pi^* \in \Delta_K$.

Intuitively, while early stochastic fluctuations influence the trajectory, the system eventually stabilizes in a composition where group proportions are self-consistent with the attraction probabilities. This framework rigorously explains the long-term stability and path-dependent dynamics of group formation under mutual attraction and size-biased growth, and provides a clear connection between the stochastic discrete-time process and its deterministic ODE approximation.

\begin{thm}
    \label{thm:analysis}

Let $\pi(t) \in \Delta_K$ evolve according to the update rule, 
$
\pi_k(t+1) = \pi_k(t) + \gamma_t \left(p_{ik}(t) - \pi_k(t)\right),
$
where $\gamma_t = \frac{1}{N_k(t)+1}$ is the step size and 
$p_{ik}(t) = \frac{\theta \pi_k(t)^{-\beta} + \epsilon_k}{\sum_{j=1}^K \theta \pi_j(t)^{-\beta} + \epsilon_j}, \quad \theta > 0, \beta > 0, \epsilon_k > 0.$
Then:
\begin{itemize}
    \item[(i)] This defines a Robbins--Monro-type stochastic approximation algorithm.
    \item[(ii)] The sequence $\pi(t)$ converges almost surely to a fixed point $\pi^* \in \Delta_K$ satisfying
    $
    \pi_k^* = p_{ik}(\pi^*) \quad \text{for all } k = 1, \dots, K.
    $
\end{itemize}
\end{thm}

\begin{proof}
We verify the standard Robbins--Monro assumptions:

\begin{asm}[Step Size Conditions]
Let $\gamma_t = \frac{1}{N_k(t)+1}$. Assuming $N_k(t) \to \infty$, then:
$
\gamma_t \to 0, \quad \sum_{t=1}^\infty \gamma_t = \infty, \quad \sum_{t=1}^\infty \gamma_t^2 < \infty.
$
This matches standard requirements for stochastic approximation.
\end{asm}

\begin{asm}[Boundedness of Iterates]
At each $t$, $\pi_k(t) \in [0,1]$ and $\sum_{k=1}^K \pi_k(t) = 1$. The update rule:
\[
\pi_k(t+1) = \frac{N_k(t)\pi_k(t) + p_{ik}(t)}{N_k(t)+1}
\]
is a convex combination of values in $[0,1]$, hence $\pi(t) \in \Delta_K$ for all $t$.
\end{asm}

\begin{asm}[Lipschitz Drift]
Define the drift function:
$
f_k(\pi) = p_{ik}(\pi) - \pi_k,
$
where
$
p_{ik}(\pi) = \frac{\theta \pi_k^{-\beta} + \epsilon_k}{\sum_{j=1}^K \theta \pi_j^{-\beta} + \epsilon_j}.
$
Assuming $\pi_k(t) > \delta > 0$ (guaranteed via $\epsilon_k > 0$), all components of $f_k(\pi)$ are smooth and Lipschitz on the interior of the simplex. Hence, $f_k$ is Lipschitz continuous.
\end{asm}

\begin{asm}[Martingale Noise (Optional)]
If $p_{ik}(t)$ is estimated via random sampling (e.g., via indicator $\tilde{p}_{ik}(t) \in \{0,1\}$ with $\mathbb{E}[\tilde{p}_{ik}(t) | \mathcal{F}_t] = p_{ik}(t)$), then the update becomes:
$
\pi_k(t+1) = \pi_k(t) + \gamma_t (\tilde{p}_{ik}(t) - \pi_k(t)).
$
The noise term $\tilde{p}_{ik}(t) - p_{ik}(t)$ is a bounded martingale difference with variance at most 1. This satisfies the standard noise conditions.
\end{asm}

By the Robbins--Monro theorem (\citep{Benaim1996, KushnerYin2003}), under these assumptions, $\pi(t)$ converges almost surely to a stable equilibrium of the limiting ODE:
$
\frac{d\pi_k}{dt} = p_{ik}(\pi) - \pi_k.
$
The fixed points of this ODE are characterized by $\pi_k^* = p_{ik}(\pi^*)$ for all $k$. If the ODE has a unique globally attractive fixed point, convergence to $\pi^*$ is guaranteed.
\end{proof}

\section{\centering Simulations}\label{sec:simu}

\begin{wrapfigure}{r}{0.5\textwidth}
\centering
\includegraphics[width=0.5\textwidth]{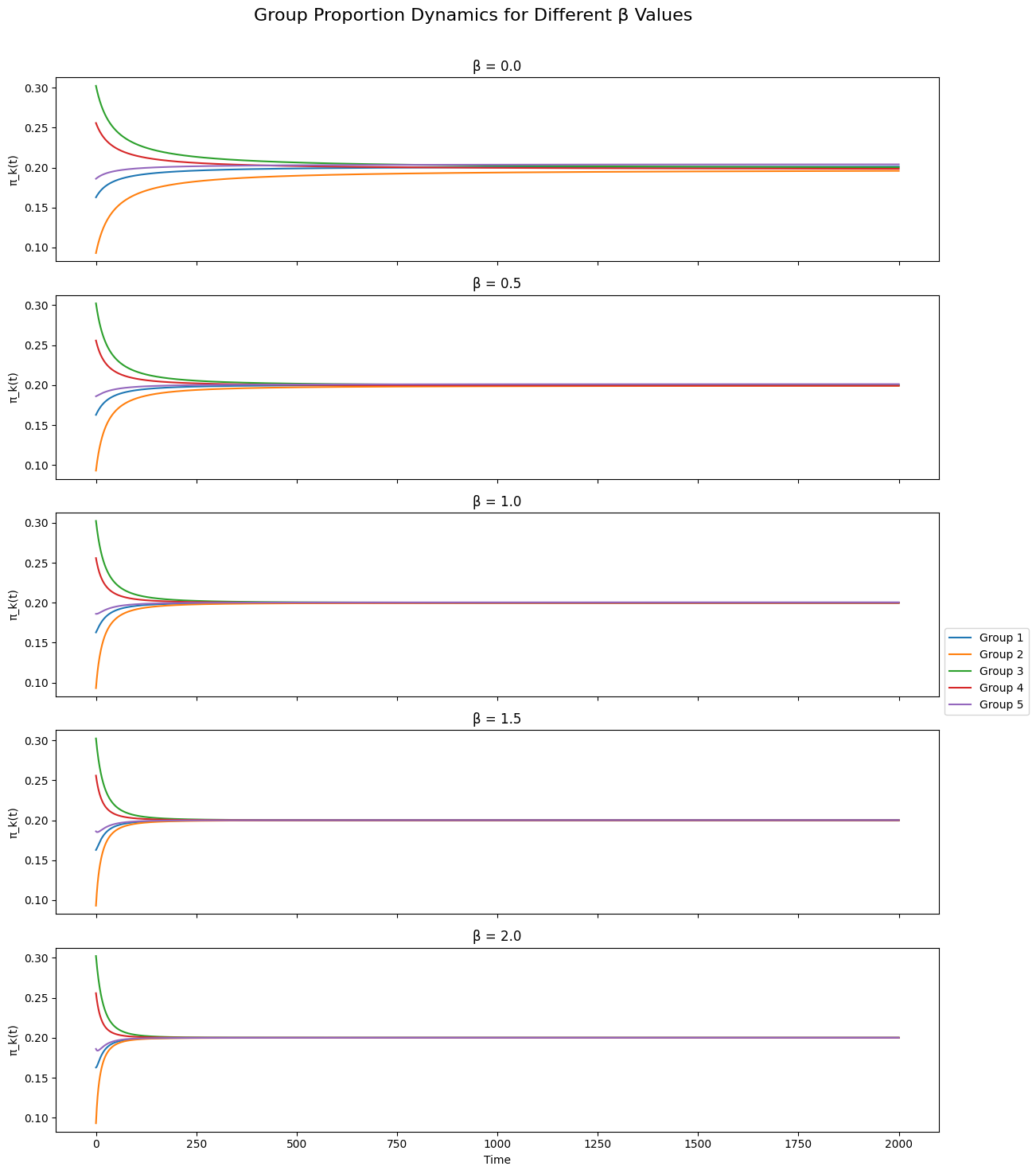}
\caption{Stability around fixed point due to symmetric Bias}
\label{fig:symmetric}
\end{wrapfigure}

To examine how the inverse-preference parameter $\beta$ affects long-term group composition, we implemented a simulation (Figure ~\ref{fig:symmetric}) that evolves group sizes over time under repeated stochastic entry. We consider a fixed number of groups ($K = 5$) and initialize group sizes randomly from a uniform integer range. At each time step, a new entrant joins the population and selects a group according to probabilistic weights derived from the attraction function, which depends on group size and the parameter $\beta$. A small group-specific noise term is added to each group’s attraction score to ensure that all groups remain viable and to simulate mild randomness or exogenous fluctuation in attractiveness.

The simulation is run for $T = 1000$ time steps for a range of $\beta$ values from $0.0$ to $2.0$. For each value of $\beta$, we record the full time series of group proportions \Cref{eq:grp_prop}, as well as the final group sizes, proportions, and the probabilities used by entrants at the last time step.

To compare behavior across $\beta$, we generate plots of Equation \Cref{eq:grp_prop} for all groups over time. These reveal how inverse-preference strength shapes convergence and dominance patterns: lower values of $\beta$ tend to reinforce initially larger groups, while higher values introduce stronger penalization for large group sizes, often promoting diversity or maintaining balance.

In addition to visualizing the dynamics, we construct summary tables capturing the initial and final group sizes, final group proportions, and final entrant probabilities for each group under each $\beta$. These statistics allow for direct numerical comparison across different model regimes and demonstrate the long-term implications of varying the strength of size-based preference in the attraction function.


\begin{wrapfigure}{l}{0.5\textwidth}
\centering
\includegraphics[width=0.5\textwidth]{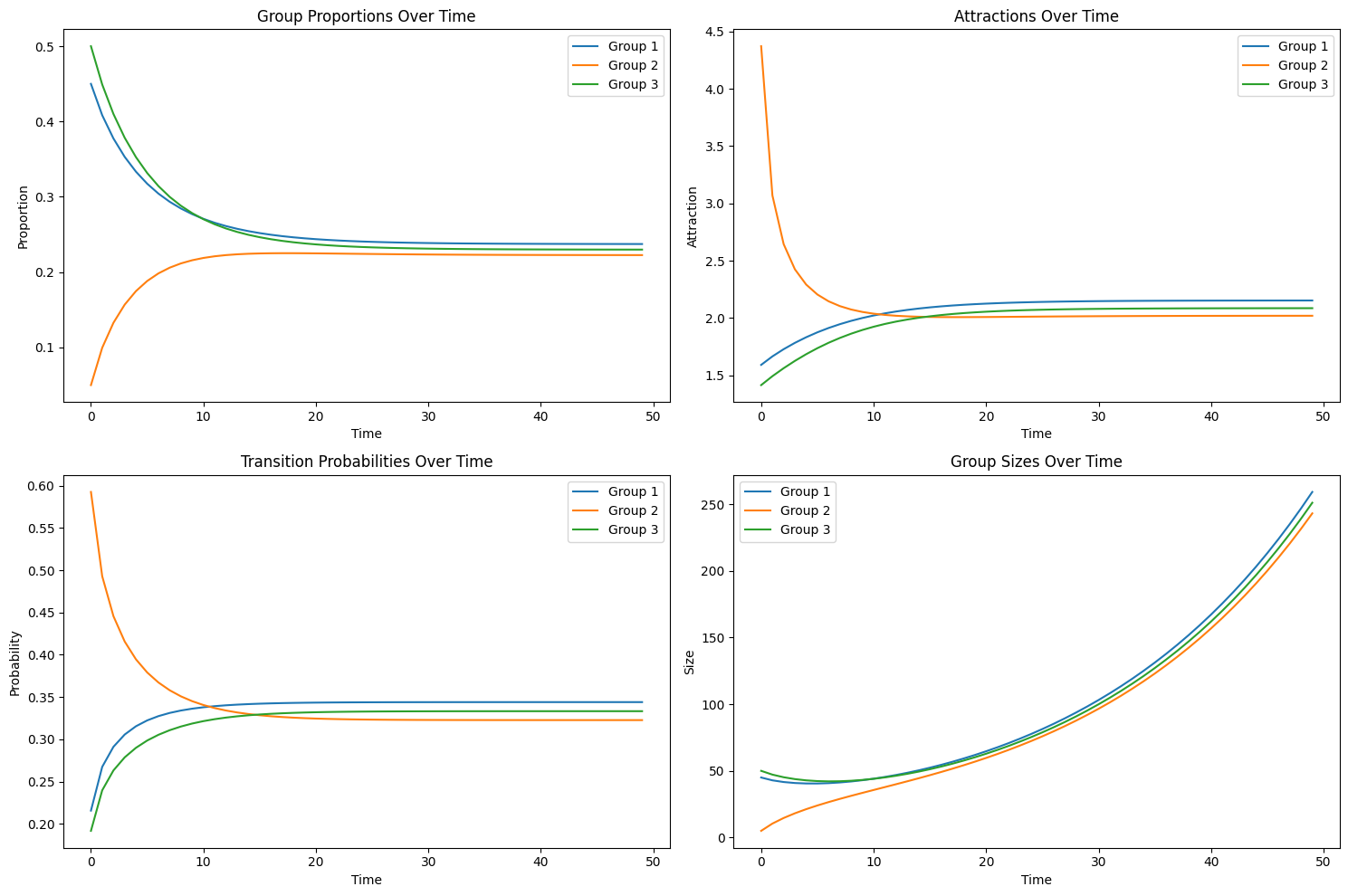}
\caption{Group Proportion vs Attraction vs Transition Probabilities vs Group Size over time}
\label{fig:GP_A_TP_S}
\end{wrapfigure}

We implemented a discrete-time simulation (Figure ~\ref{fig:GP_A_TP_S}.) to examine the evolution of group proportions under the biased attraction model defined in the previous section. Starting with a fixed initial population distributed across \(K\) groups, the simulation updates group sizes iteratively by incorporating both external growth and internal redistribution. At each time step, the proportion of individuals in each group is used to compute its attraction, which includes an inverse preference term and group-specific bias (see Equation \Cref{eq:Attraction_Potential}). These attraction values are normalized to yield transition probabilities (see Equation \Cref{eq:choosing_prob}) that determine how new individuals and existing members select groups. A fraction \(\eta\) of the current population is introduced as new entrants, allocated across groups according to these transition probabilities, while the remaining population is allowed to switch groups in proportion to the difference between desired and actual group sizes. To ensure numerical stability and avoid abrupt changes, a damping factor is applied to smooth the update of group sizes (see Equation \Cref{eq:grp_prop}). 

\begin{wrapfigure}{r}{0.5\textwidth}
\centering
\includegraphics[width=0.5\textwidth]{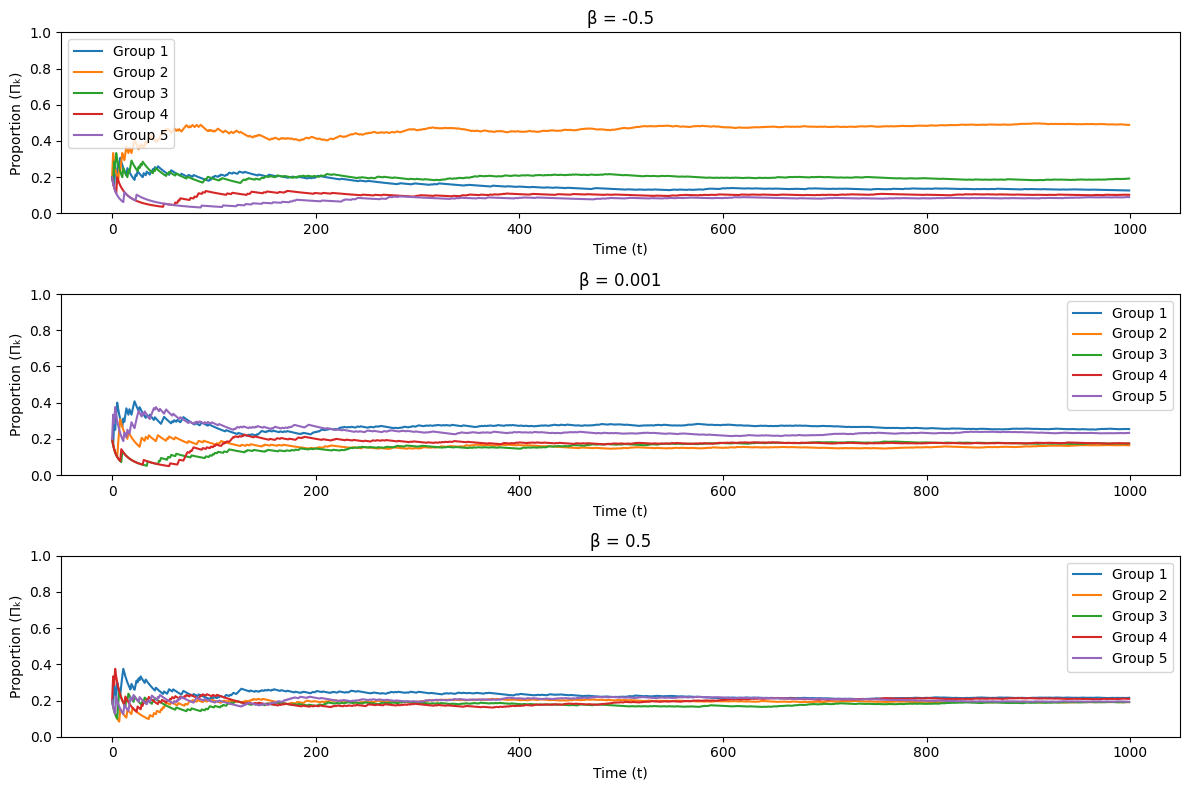}
\caption{Stability around fixed point due to asymmetric bias}
\label{fig:Equilibrium}
\end{wrapfigure}

Throughout the simulation, we record the evolution of group sizes, proportions, attraction values, and transition probabilities, which together provide insight into how the system evolves and stabilizes under various parameter configurations. This simulation framework supports asymmetric biases and varying preference strength, allowing us to explore convergence behavior, persistent diversity, and the influence of early conditions on long-run group composition.
The curves in the simulation exhibit noticeable fluctuations (Figure ~\ref{fig:Equilibrium}) and lack smoothness primarily due to the inherent stochasticity in the group selection process. At each time step, the assignment of a new member to a group is probabilistic, which introduces random variation in group sizes and proportions over time. This randomness is particularly evident when group proportions are small or when attraction probabilities are similar, causing the proportions to fluctuate rather than evolve smoothly. Moreover, the different values of the bias parameter \(\beta\) — specifically \(-0.5\), \(0.1\), and \(0.5\) — significantly influence the system's behavior. Negative \(\beta\) values tend to favor larger groups, positive values favor smaller groups, and near-zero values imply minimal size bias. These nonlinear effects, combined with stochasticity, lead to bifurcation phenomena where the system may converge to distinct equilibrium configurations depending on initial conditions and random fluctuations. Consequently, the interaction of different \(\beta\) parameters with the stochastic dynamics produces irregular, non-smooth trajectories and multiple stable equilibria observed in the simulation.

\begin{wrapfigure}{r}{0.5\textwidth}
\centering
\includegraphics[width=0.5\textwidth]{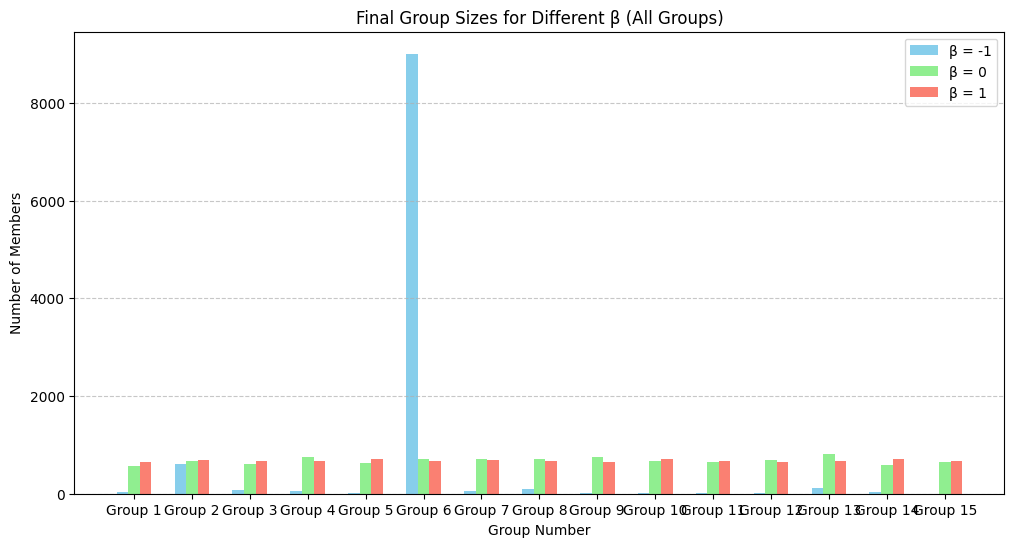}
\caption{Group size for different $\beta$'s}
\label{fig:grp_form}
\end{wrapfigure}

In this extended simulation study, the system exhibits stability around the fixed point. Figure~\ref{fig:grp_form} shows the group formation for different values of $\beta$. We analyze the evolution of group proportions within a system of \(K=15\) groups over \(T=10{,}000\) discrete time steps. Unlike the previous uniform initialization, each simulation run begins with randomly assigned initial group sizes, drawn uniformly between 1 and 20 members, to capture more heterogeneous starting conditions. Three different values of the bias parameter \(\beta \in \{-1, 0, 1\}\) are tested to explore the influence of size bias on group dynamics. At each time step, the probability of a new member joining a given group depends on a nonlinear mutual attraction matrix \(M\), computed similarly to the prior model, and scaled by the group's current proportion raised to the power \(-\beta\). This framework allows us to model varying preferences: \(\beta = -1\) favors larger groups, \(\beta = 1\) favors smaller groups, and \(\beta = 0\) corresponds to neutral size bias. The group selection process is stochastic, with each new member assigned probabilistically based on normalized attraction scores. The simulation records the temporal evolution of group proportions, which are plotted to observe convergence behavior and potential equilibrium states. Additionally, a bar plot of the final group sizes across all groups for each \(\beta\) illustrates the long-term impact of size bias under heterogeneous initial conditions. This approach provides insights into how initial heterogeneity and size preference jointly shape group composition over extended time horizons.



\begin{wrapfigure}{l}{0.5\textwidth}
\centering
\includegraphics[width=0.5\textwidth]{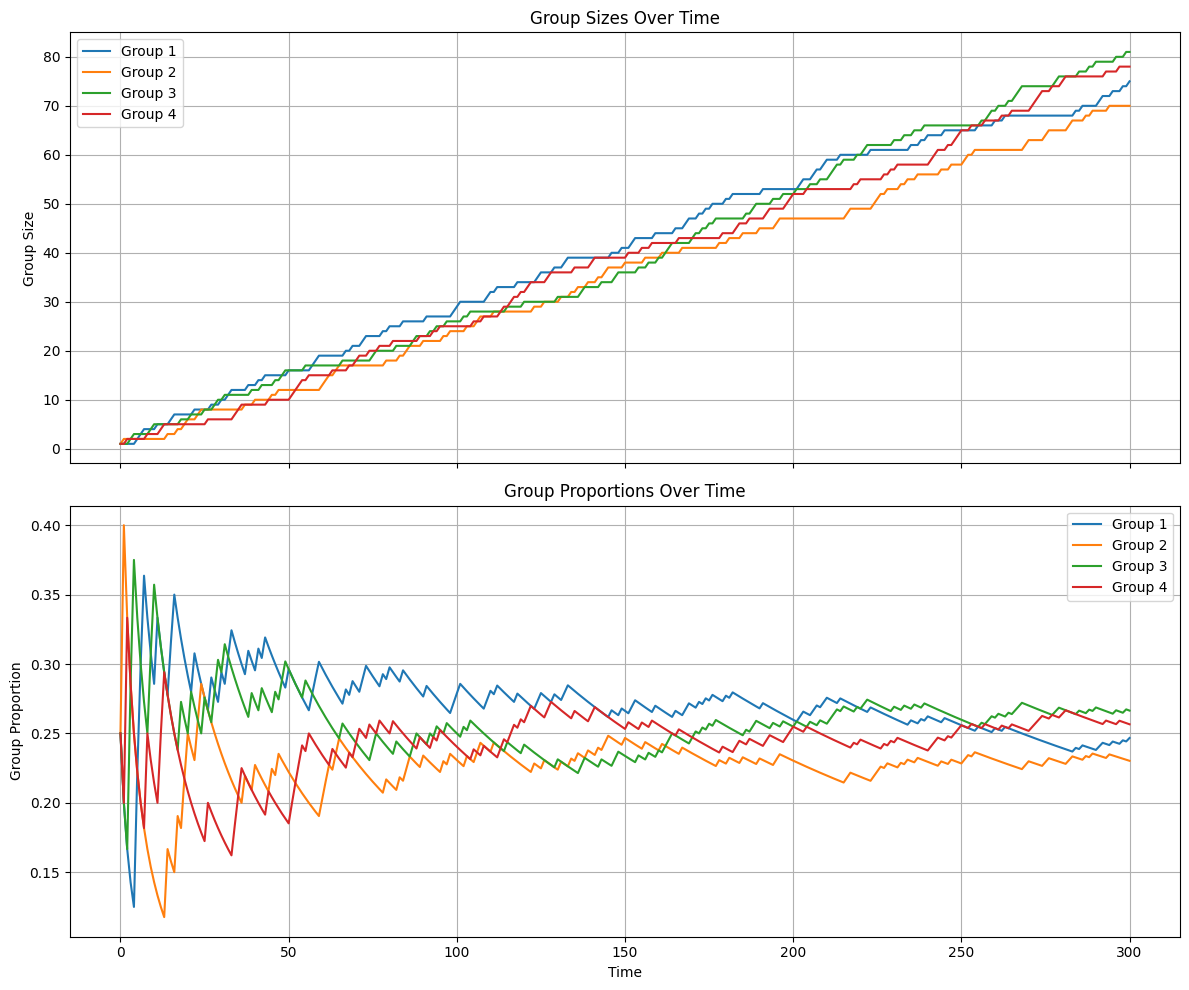}
\caption{Group formation and their proportion over time}
\label{fig:G_form_and_Prop}
\end{wrapfigure}

We performed a simulation study of group dynamics (Figure ~\ref{fig:G_form_and_Prop}. for \(K=10\) groups over \(T=3000\) time steps, analyzing how different values of the bias parameter \(\beta\) influence group growth and final size distributions. Each simulation starts with a heterogeneous initial distribution of group members, reflecting realistic variability at \(t=0\). The group sizes evolve as new members probabilistically join groups based on a mutual attraction model weighted by group proportions raised to the power \(-\beta\). This parameter \(\beta\) modulates size bias: negative \(\beta\) values favor larger groups, positive \(\beta\) values favor smaller groups, and values near zero indicate minimal bias.

For \(\beta = -0.5\), the simulation reveals a strong preferential attachment to already large groups, resulting in a highly skewed final distribution where a few groups dominate the population (e.g., one group grows to over 5500 members while others remain below 1100). Conversely, for \(\beta = 0.1\), the final group sizes are more balanced, with all groups reaching roughly similar populations around 1400–1600 members. For \(\beta = 0.5\), the bias toward smaller groups further equalizes the group sizes, yielding nearly uniform final group sizes (Table ~\ref{tab:group_comparison}.) across all groups.

These results demonstrate that the bias parameter \(\beta\) critically shapes group size dynamics, ranging from strong concentration under negative bias to near-equal group sizes under positive bias, mediated by the nonlinear mutual attraction mechanism and initial group heterogeneity.

\begin{table}[htb!]
\centering
\begin{tabular}{|c|cc|cc|cc|}
\toprule
\multirow{2}{*}{Group} & \multicolumn{2}{c|}{$\beta = -0.5$} & \multicolumn{2}{c|}{$\beta = 0.1$} & \multicolumn{2}{c|}{$\beta = 0.5$} \\
\cmidrule(lr){2-3}\cmidrule(lr){4-5}\cmidrule(lr){6-7}
 & Initial & Final & Initial & Final & Initial & Final \\
\midrule
1  & 7    & \textcolor{blue}{2829} & 4    & 1535 & 5    & 1504 \\
2  & 4    & 746  & 8    & 1348 & 1    & 1537 \\
3  & 8    & 976  & 8    & 1461 & 10   & 1532 \\
4  & 5    & 797  & 3    & 1527 & 6    & 1514 \\
5  & 7    & 675  & 6    & 1498 & 9    & 1492 \\
6  & 10   & \textcolor{blue}{5561} & 5    & 1536 & 1    & 1496 \\
7  & 3    & 626  & 2    & 1618 & 10   & 1517 \\
8  & 7    & 946  & 8    & 1580 & 3    & 1451 \\
9  & 8    & 1048 & 6    & 1525 & 7    & 1486 \\
10 & 5    & 860  & 2    & 1424 & 4    & 1527 \\
\bottomrule

\end{tabular}
\caption{Group Size Comparison Across Group's for Different $\beta$ Values}
\label{tab:group_comparison}
\end{table}
\section{\centering Summary}\label{sec:summ}

This work introduces a novel probabilistic model of group dynamics integrating nonlinear mutual attraction and size-based bias controlled by a tunable parameter \(\beta\). Theoretically, the model captures rich behaviors including preferential attachment, equilibrium convergence, and bifurcation phenomena influenced by the interplay of mutual group attractiveness and group size effects. Through extensive simulations with varying \(\beta\) values and heterogeneous initial conditions, we demonstrate how the parameter modulates group size distributions—from strong concentration favoring dominant groups at negative \(\beta\) to more balanced or uniform distributions at positive \(\beta\). These insights provide a versatile framework for understanding the evolution of social, biological, or economic groups where size and mutual influence affect growth dynamics. Potential applications span network formation, market share evolution, opinion dynamics, and community detection. Future work could explore analytical characterization of equilibrium states, extensions to time-varying or adaptive \(\beta\), incorporation of exogenous shocks, and empirical validation on real-world datasets.

The nobility of this model lies in its ability to realistically capture group dynamics through a mathematically rigorous and flexible framework. By incorporating mutual attraction and a tunable parameter $\beta$that adjusts preferences for smaller or larger groups, the model reflects behaviors commonly observed in social, organizational, and economic settings. Its structure allows for dynamic analysis of group formation and evolution over time, while ensuring all groups retain nonzero selection probabilities. This versatility and theoretical depth make the model broadly applicable across disciplines, offering valuable insights into how group proportions stabilize or shift under various conditions.

\section{\centering Limitations}\label{sec:lim}

One limitation of the current simulation framework is that it models group evolution solely through the addition of new entrants, without allowing existing members to switch between groups. That is, the only mechanism driving change in group composition is the sequential entry of new individuals, whose group choice is governed by the attraction-based probabilities. As a result, the simulation does not account for mutual exchange or migration between groups among the existing population. In realistic social or organizational systems, however, individuals often reevaluate their affiliations and may switch groups based on perceived attractiveness, peer influence, or shifting group dynamics. The absence of such endogenous redistribution may cause the model to overemphasize the role of initial conditions or accumulate imbalances that would otherwise be corrected through member switching. Incorporating mutual exchange among existing members would enrich the dynamics and potentially lead to more realistic equilibrium behaviors, particularly in settings where mobility between groups is common.

\appendix
\section{}
\subsection{Degeneracy}

\begin{thm}\label{prop:concavity}
Let
\[
M(x,y)=\frac{x^{2}}{y}-x+y , \qquad (x,y)\in(0,1)^{2}.
\]
Then $M$ is convex but not strictly convex on $(0,1)^{2}$. Its Hessian is positive semidefinite of rank one at every interior point. In particular, $M$ has no interior critical points in $(0,1)^{2}$.
\end{thm}

\begin{proof}
\textit{First-order partial derivatives:}
We compute
\[
M_{x}=\frac{2x}{y}-1, 
\qquad 
M_{y}=1-\frac{x^{2}}{y^{2}}.
\]
Hence,
\[
M_{x}
\begin{cases}
<0 &\text{if } x<\frac{y}{2},\\[4pt]
=0 &\text{if } x=\frac{y}{2},\\[4pt]
>0 &\text{if } x>\frac{y}{2},
\end{cases}
\qquad
M_{y}
\begin{cases}
>0 &\text{if } x<y,\\[4pt]
=0 &\text{if } x=y,\\[4pt]
<0 &\text{if } x>y.
\end{cases}
\]
\textit{Critical points:}
Interior stationary points must satisfy
\[
M_{x}=0, \qquad M_{y}=0.
\]
These yield $x = y/2$ and $x = y$, respectively, implying $x=y=0$, which lies on the boundary of $[0,1]^{2}$. Thus, $M$ has no critical point in the interior~$(0,1)^2$.\\
\textit{Second-order derivatives and Hessian:}
We compute
\[
M_{xx}=\frac{2}{y},\qquad 
M_{xy}=M_{yx}=-\frac{2x}{y^{2}},\qquad 
M_{yy}=\frac{2x^{2}}{y^{3}}.
\]
Thus the Hessian is
\[
H(x,y)=
\begin{pmatrix}
\frac{2}{y} & -\frac{2x}{y^{2}} \\[6pt]
-\frac{2x}{y^{2}} & \frac{2x^{2}}{y^{3}}
\end{pmatrix}
=
\frac{2}{y^{3}}
\begin{pmatrix}
y \\[4pt]
- x
\end{pmatrix}
\begin{pmatrix}
y & -x
\end{pmatrix}.
\]
This factorization shows that $H(x,y)$ is positive semidefinite, since for any $v\in\mathbb{R}^{2}$,
$
v^{\top}H(x,y)v
=\frac{2}{y^{3}}\bigl(v_{1}y - v_{2}x\bigr)^{2} \ge 0.
$
Moreover, $\det(H)=0$ everywhere on $(0,1)^2$, so the Hessian has rank one and is never positive definite. Therefore $M$ is convex but not strictly convex. Since $M$ is convex (but not strictly convex) and has no interior critical points.
\end{proof}

\subsection{Equilibrium Conditions}
\begin{proof} [Proof of Theorem~\ref{thm:equilibrum}]

Let \( \chi_k(t) \) be the indicator random variable for whether the entrant at time \( t \) chooses group \( k \):
\[
\chi_k(t) =
\begin{cases}
1, & \text{if group } k \text{ is chosen} \\
0, & \text{otherwise}
\end{cases}
\]
Then the expected value \cite{ross2014introduction} of \( \chi_k(t) \) is simply the probability of choosing group \( k \): $\mathbb{E}[\chi_k(t)] = p_{k}(t)$\\
The group size \( n_k(t) \) is updated as:
\[
n_k(t+1) = n_k(t) + \chi_k(t)
\]
Taking expectation on both sides gives:
\[
\mathbb{E}[n_k(t+1)] = n_k(t) + \mathbb{E}[\chi_k(t)] = n_k(t) + p_{k}(t) \]
This gives the expected update rule:
\[
n_k(t+1) = n_k(t) + p_{k}(t) \]

This allows the study of the average system behavior, such as convergence and equilibrium, without tracking individual random events.
\[
n_k(t+1) = n_k(t) + p_{k}(t), \qquad N(t+1) = N(t) + 1
\]
Thus, the group proportion evolves as:
\[
\pi_k(t+1) = \frac{n_k(t) + p_{k}(t)}{N(t) + 1} = \frac{N(t)\pi_k(t) + p_{k}(t)}{N(t) + 1}
\]
We can rewrite this as:
\[
\pi_k(t+1) = \pi_k(t) + \frac{p_{k}(t) - \pi_k(t)}{N(t) + 1}
\]
This has the form of a stochastic approximation:
\[
\pi_k(t+1) = \pi_k(t) + \eta_t (p_{k}(t) - \pi_k(t)), \quad \eta_t = \frac{1}{N(t) + 1}
\]
Which approximates the differential equation:
$
\frac{d\pi_k}{dt} = p_{k}(t) - \pi_k
$
We seek a fixed point $\pi_k^*$ such that:
\[
\frac{d\pi_k}{dt} = 0 \quad \Rightarrow \quad p_{k}(t) = \pi_k^*
\]
Using the definition of $p_{k}$, this becomes:
$
\pi_k^* = \frac{\theta \pi_k^{*-\beta} + \epsilon_k}{\sum\limits_{j=1}^K (\theta \pi_j^{*-\beta} + \epsilon_j)}
$\\
\boxed{Assume\,symmetry}: $\epsilon_k = \epsilon$ for all $k$.
$\\
\text{LHS: } \pi_k^* = \frac{1}{K}
$\\
$
\text{RHS: } \frac{\theta \left(\frac{1}{K}\right)^{-\beta} + \epsilon}{K \cdot \theta \left(\frac{1}{K}\right)^{-\beta} + K\epsilon}
= \frac{\theta K^\beta + \epsilon}{K (\theta K^\beta + \epsilon)} = \frac{1}{K}
$.\\
So $\pi_k^* = \frac{1}{K}$ is a valid fixed point.
The proportions $\pi_k(t) = \frac{n_k(t)}{N(t)}$ follow a stochastic process that converges to a fixed point. Under symmetric $\epsilon_k = \epsilon$, the unique fixed point is:
    \[
    \pi_k^* = \frac{1}{K}, \quad \forall k
    \]
\boxed{Assume \, asymmetry:} The equilibrium is globally stable: if perturbed, the system returns to uniform distribution. Now let $\epsilon_k$ differ across groups. Then:
\[
p_{k}(\pi) = \frac{\theta \pi_k^{-\beta} + \epsilon_k}{\sum_{j=1}^K (\theta \pi_j^{-\beta} + \epsilon_j)}.
\]
At equilibrium, we solve for $\pi_k^*$ such that:
\[
\pi_k^* = p_{k}(\pi^*) = \frac{\theta \pi_k^{* -\beta} + \epsilon_k}{\sum_{j=1}^K (\theta \pi_j^{* -\beta} + \epsilon_j)}.
\]
Let $S = \sum\limits_{j=1}^K (\theta \pi_j^{* -\beta} + \epsilon_j)$. Then:
$
\pi_k^* S = \theta \pi_k^{* -\beta} + \epsilon_k \quad \text{implies} \quad \pi_k^* S - \theta \pi_k^{*-\beta} = \epsilon_k.
$
This is a system of $K$ nonlinear equations in $\pi_1^*, \dots, \pi_k^*$. It can be solved numerically.\\
Let analyze First-Order Approximation of Stationary Distribution under asymmetric \(\epsilon_k\). We consider a system of \( K \) groups where each group has a time-evolving proportion \( \pi_k(t) \), and entrants choose a group based on the probability is Equation \Cref{eq:choosing_prob} with \( \beta > 0 \), and group-specific bias terms \( \epsilon_k > 0 \). We aim to approximate the stationary distribution \( \pi_k^* \) under small perturbations in \( \epsilon_k \).
Assume a small perturbation around the symmetric equilibrium:
\[
\pi_k^* = \frac{1}{K} + \delta_k, \quad \text{with } \sum_{k=1}^K \delta_k = 0
\]
\[
\epsilon_k = \epsilon + \eta_k, \quad \text{with } \sum_{k=1}^K \eta_k = 0
\]
We expand the fixed point equation around \( \pi_k^* = \frac{1}{K} \).
The fixed point satisfies: $\pi_k^* \cdot S = \theta \pi_k^{*-\beta} + \epsilon_k$
where $ S = \sum\limits_{j=1}^K \left[ \theta \pi_j^{*-\beta} + \epsilon_j \right]$.
We expand \( \pi_k^{*-\beta} \) using a first-order Taylor approximation:
\[
\pi_k^{*-\beta} = \left( \frac{1}{K} + \delta_k \right)^{-\beta}
\approx K^{\beta} (1 - \beta K \delta_k)
\]
Next, we expand \( S \):
$= \sum\limits_{j=1}^K \left[ \theta \pi_j^{*-\beta} + \epsilon_j \right]= \sum\limits_{j=1}^K \left[ \theta K^{\beta} (1 - \beta K \delta_j) + \epsilon + \eta_j \right]
= K(\theta K^{\beta} + \epsilon)$\\
Left-hand side (LHS):
$
\pi_k^* S = \left( \frac{1}{K} + \delta_k \right)(\theta K^{\beta+1} + K \epsilon)
= \theta K^{\beta} + \epsilon + \theta K^{\beta+1} \delta_k + K \epsilon \delta_k
$\\
Right-hand side (RHS):
$
\theta \pi_k^{*-\beta} + \epsilon_k \approx \theta K^{\beta}(1 - \beta K \delta_k) + \epsilon + \eta_k
= \theta K^{\beta} - \theta \beta K^{\beta+1} \delta_k + \epsilon + \eta_k
$\\
Subtracting LHS - RHS gives:
\[
\delta_k = \frac{\eta_k}{\theta K^{\beta+1}(1 + \beta) + K \epsilon}
\]
The first-order approximation of the equilibrium proportion is:
$
\pi_k^* = \frac{1}{K} + \frac{\eta_k}{\theta K^{\beta+1}(1 + \beta) + K \epsilon}
$\\
that is:
$
\pi_k^* =\frac{1}{K}+\frac{\eta_k}{\theta (\beta+1)} K^{-(\beta+1)}+O(K^{-(2\beta+1)})
$
This shows that groups with higher \( \epsilon_k \) values have proportionally larger long-run group shares, and the influence is modulated by \( \theta \), \( \beta \), \( \epsilon \), and the number of groups \( K \).
\end{proof}

\subsection{Smooth Solution:}
\begin{thm}
    \label{eq:analytic_Sol}
    Let the differential equation is: $\frac{d\pi_k}{dt} = p_{k}(\pi(t)) - \pi_k(t)$.
The solution can be decomposed using the eigendecomposition of $(J - I)$: $\pi_k(t)=\pi^*(t)+y_k(t)$
where \, $y_k(t) = \sum\limits_{i=1}^n c_i v_k^{(i)} e^{\lambda_i t}$
with $\lambda_i$ are the eigenvalues, $v^{(i)}$ are the associated eigenvectors, and $c_i$ are constants determined by the initial perturbation $\mathbf{y}(0)$.
\end{thm}
\begin{proof}

The given differential equation is: $\frac{d\pi_k}{dt} = p_{k}(\pi(t)) - \pi_k(t)$
Here, \( \pi(t) = (\pi_1(t), \pi_2(t), \dots, \pi_n(t)) \) is a vector-valued function, and \( p_{k}(\pi(t)) \) is some function of \( \pi(t) \). This is a first-order ordinary differential equation (ODE). 
To solve it, we can rewrite it in the standard linear form:
$
\frac{d\pi_k}{dt} + \pi_k(t) = p_{k}(\pi(t))
$
This is a linear ODE in \( \pi_k(t) \), but since \( p_{k}(\pi(t)) \) may depend on other components of \( \pi(t) \), the system might be coupled (i.e., the equations for different \( \pi_k \) may depend on each other). 
Consider the coupled system of ordinary differential equations (ODEs):
\[
\frac{d\pi_k}{dt} = p_{k}(\pi(t)) - \pi_k(t), \quad k = 1, 2, \dots, n
\]
where \(\pi(t) = (\pi_1(t), \pi_2(t), \dots, \pi_k(t))\) and each \(p_{k}(\pi(t))\) depends on the full state \(\pi(t)\). \\
Let \(\pi(t) = \begin{bmatrix} \pi_1(t) \\ \pi_2(t) \\ \vdots \\ \pi_n(t) \end{bmatrix}\) and \(\bm{p}(\pi) = \begin{bmatrix} p_1(\pi) \\ p_2(\pi) \\ \vdots \\ p_n(\pi) \end{bmatrix}\). 
The system becomes:
$
\frac{d\pi}{dt} = \bm{p}(\pi) - \pi
$
This is a nonlinear coupled ODE system because \(\bm{p}(\pi)\) depends on \(\pi\).
Find the equilibrium points \(\pi^*\) where \(\frac{d\pi}{dt} = \bm{0}\):
\[
\bm{p}(\pi^*) - \pi^* = \bm{0} \implies \pi^* = \bm{p}(\pi^*)
\]
This is a fixed-point equation. Numerical methods (e.g., Newton-Raphson) may be required if no analytical solution exists.
For small perturbations \(\bm{y}(t)\) around \(\pi^*\), let \(\pi(t) = \pi^* + \bm{y}(t)\). Linearizing $\bm{p}(\pi) \approx \bm{p}(\pi^*) + J(\pi^*) \bm{y}$, where \(J(\pi^*)\) is the Jacobian matrix of \(\bm{p}\) evaluated at \(\pi^*\): $
J_{ij} = \left. \frac{\partial p_i}{\partial \pi_j} \right|_{\pi = \pi^*}$. Substituting into the ODE:
\[
\frac{d\bm{y}}{dt} = \bm{p}(\pi^*) + J \bm{y} - (\pi^* + \bm{y}) = (J - I) \bm{y}
\]
The solution to this linear system is:
$
\bm{y}(t) = e^{(J - I)t} \bm{y}(0)
$
The solution can be decomposed using the eigendecomposition of $(J - I)$:
$
\pi_k(t)=\pi^*(t)+y_k(t)\, \text{where} \, y_k(t) = \sum_{i=1}^n c_i v_k^{(i)} e^{\lambda_i t},
$
where $\lambda_i$ are the eigenvalues, $v^{(i)}$ are the associated eigenvectors, and $c_i$ are constants determined by the initial perturbation $\mathbf{y}(0)$.
\end{proof}

\bibliographystyle{plain}
\bibliography{refrences}

@article{schelling1971,
  title={Dynamic models of segregation},
  author={Schelling, Thomas C.},
  journal={Journal of Mathematical Sociology},
  volume={1},
  number={2},
  pages={143--186},
  year={1971},
  publisher={Taylor \& Francis}
}

@article{barabasi1999,
  title={Emergence of scaling in random networks},
  author={Barab{\'a}si, Albert-L{\'a}szl{\'o} and Albert, R{\'e}ka},
  journal={Science},
  volume={286},
  number={5439},
  pages={509--512},
  year={1999},
  publisher={American Association for the Advancement of Science}
}

@article{degroot1974,
  title={Reaching a consensus},
  author={DeGroot, Morris H.},
  journal={Journal of the American Statistical Association},
  volume={69},
  number={345},
  pages={118--121},
  year={1974},
  publisher={Taylor \& Francis}
}

@article{taylor1978,
  title={Evolutionary stable strategies and game dynamics},
  author={Taylor, Peter D. and Jonker, Leo B.},
  journal={Mathematical Biosciences},
  volume={40},
  number={1-2},
  pages={145--156},
  year={1978},
  publisher={Elsevier}
}

@article{hegselmann2002,
  title={Opinion dynamics and bounded confidence models, analysis, and simulation},
  author={Hegselmann, Rainer and Krause, Ulrich},
  journal={Journal of Artificial Societies and Social Simulation},
  volume={5},
  number={3},
  year={2002},
  note={Available at: \url{http://jasss.soc.surrey.ac.uk/5/3/2.html}}
}

@article{bikhchandani1992,
  title={A theory of fads, fashion, custom, and cultural change as informational cascades},
  author={Bikhchandani, Sushil and Hirshleifer, David and Welch, Ivo},
  journal={Journal of Political Economy},
  volume={100},
  number={5},
  pages={992--1026},
  year={1992},
  publisher={University of Chicago Press}
}

@book{ross2014introduction,
  title={Introduction to probability models},
  author={Ross, Sheldon M},
  year={2014},
  publisher={Academic press}
}

@article{zabarianska2024,
  author = {Zabarianska, Iryna and Proskurnikov, Anton V.},
  title = {Coalescing Force of Group Pressure: Consensus in Nonlinear Opinion Dynamics},
  year = {2024},
  journal = {arXiv preprint},
}

@article{sampson2024,
  author = {Sampson, Corbit R. and Porter, Mason A. and Restrepo, Juan G.},
  title = {Oscillatory and Excitable Dynamics in an Opinion Model with Group Opinions},
  year = {2024},
  journal = {arXiv preprint},
}

@article{pasimeni2025,
  author = {Pasimeni, Francesco and Wade, R. and Alkemade, F.},
  title = {Opinion Dynamic and Social Clustering in a 2D Space: An Agent Based Experiment},
  year = {2025},
  journal = {Computational Economics},
}

@article{albi2025,
  author = {Albi, Giacomo and Calzola, E. and Dimarco, G.},
  title = {A data-driven kinetic model for opinion dynamics with social network contacts},
  year = {2025},
  journal = {European Journal of Applied Mathematics},
}

@incollection{peralta2023,
  author = {Peralta, A. F. and Kertész, J. and {\'I}\~{n}iguez, G.},
  title = {Opinion dynamics in social networks: From models to data},
  year = {2023},
  booktitle = {Handbook of Computational Social Science},
  editor = {Yasseri, T.},
  publisher = {Edward Elgar Publishing},
}

@article{frontiers2025,
  author = {Anonymous},
  title = {Modeling social conformity and peer pressure in opinion dynamics: the role of dynamic interaction structures},
  year = {2025},
  journal = {Frontiers in Physics},
}

@book{sandholm2010population,
  title={Population Games and Evolutionary Dynamics},
  author={Sandholm, William H.},
  publisher={MIT Press},
  year={2010}
}

@article{pemantle2007survey,
  title={A survey of random processes with reinforcement},
  author={Pemantle, Robin},
  journal={Probability Surveys},
  volume={4},
  pages={1--79},
  year={2007}
}

@article{bandyopadhyay2018negative,
  title={Generalized Pólya urn schemes with negative but linear reinforcements},
  author={Bandyopadhyay, Antar and Kaur, Gursharn},
  journal={arXiv preprint arXiv:1809.08506},
  year={2018}
}

@phdthesis{kaur2019negatively,
  title={Negatively Reinforced Balanced Urn Schemes},
  author={Kaur, Gursharn},
  school={Indian Statistical Institute},
  year={2019}
}

@article{deffuant2000mixing,
  title={Mixing beliefs among interacting agents},
  author={Deffuant, Guillaume and Neau, David and Amblard, Frédéric and Weisbuch, Gérard},
  journal={Advances in Complex Systems},
  volume={3},
  number={1-4},
  pages={87--98},
  year={2000}
}

@article{hegselmann2002opinion,
  title={Opinion dynamics and bounded confidence models},
  author={Hegselmann, Rainer and Krause, Ulrich},
  journal={Journal of Artificial Societies and Social Simulation},
  volume={5},
  number={3},
  year={2002}
}

@book{friedkin2011social,
  title={Social Influence Network Theory: A Sociological Examination of Small Group Dynamics},
  author={Friedkin, Noah E. and Johnsen, Eugene C.},
  publisher={Cambridge University Press},
  year={2011}
}

@article{newman2006modularity,
  title={Modularity and community structure in networks},
  author={Newman, M. E. J.},
  journal={Proceedings of the National Academy of Sciences},
  volume={103},
  number={23},
  pages={8577--8582},
  year={2006}
}

@article{mas2013differentiation,
  title={Differentiation without distancing: Explaining bi-polarization of opinions},
  author={M{\"a}s, Michael and Flache, Andreas},
  journal={PLoS ONE},
  volume={8},
  number={11},
  pages={e74516},
  year={2013}
}

@article{christie2023nfds,
  title={Negative frequency-dependent selection: mechanisms and consequences},
  author={Christie, Mark R.},
  journal={Biological Reviews},
  year={2023}
}

@article{siew2020inverse,
  title={Investigating the influence of inverse preferential attachment},
  author={Siew, Cynthia S. Q. and Vitevitch, Michael S. and others},
  journal={Physica A: Statistical Mechanics and its Applications},
  volume={556},
  pages={124833},
  year={2020}
}

@article{mcfadden1974conditional,
  title={Conditional logit analysis of qualitative choice behavior},
  author={McFadden, Daniel},
  journal={Frontiers in Econometrics},
  pages={105--142},
  year={1974}
}

@article{mckelvey1995quantal,
  title={Quantal response equilibria for normal form games},
  author={McKelvey, Richard D. and Palfrey, Thomas R.},
  journal={Games and Economic Behavior},
  volume={10},
  number={1},
  pages={6--38},
  year={1995}
}

@article{blume1993statistical,
  title={The statistical mechanics of strategic interaction},
  author={Blume, Lawrence E.},
  journal={Games and Economic Behavior},
  volume={5},
  number={3},
  pages={387--424},
  year={1993}
}

@article{Benaim1996,
  author = {Bena\"{\i}m, Michel},
  title = {A dynamical system approach to stochastic approximations},
  journal = {SIAM Journal on Control and Optimization},
  year = {1996},
  volume = {34},
  number = {2},
  pages = {437--472},
  doi = {10.1137/S0363012993253534}
}

@book{KushnerYin2003,
  author = {Kushner, Harold J. and Yin, G. George},
  title = {Stochastic Approximation and Recursive Algorithms and Applications},
  year = {2003},
  edition = {2nd},
  publisher = {Springer},
  series = {Stochastic Modelling and Applied Probability, Vol.35},
  isbn = {9780387008943}
}

\end{document}